\documentclass[a4paper,oneside,11pt, reqno]{amsproc}
\usepackage{amsthm}
\usepackage{amsfonts}
\usepackage{amsmath}
\usepackage{amssymb}
\usepackage{mathrsfs}
\usepackage{epsfig}
\usepackage{graphicx}
\usepackage{epstopdf}
\usepackage{tikz-cd}
\usepackage{color}
\usepackage{ifthen}
\usepackage{float}
\usepackage{enumitem}

\makeatletter
\@namedef{subjclassname@2010}{%
\textup{2010} Mathematics Subject Classification}
\makeatother

\newtheorem{theorem}{Theorem}[section]
\newtheorem{proposition}[theorem]{Proposition}%[section]
\newtheorem{lemma}[theorem]{Lemma}%[section]

\theoremstyle{remark}
\newtheorem{remark}[theorem]{Remark}%[section]

\newcommand\Item[1][]{%
  \ifx\relax#1\relax  \item \else \item[#1] \fi
  \abovedisplayskip=0pt\abovedisplayshortskip=0pt~\vspace*{-\baselineskip}}

\theoremstyle{definition}
\newtheorem{definition}[theorem]{Definition}%[section]

\newcommand{\bq}{\begin{equation}}
\newcommand{\eq}{\end{equation}}
\newcommand{\beqn}{\begin{eqnarray*}}
\newcommand{\eeqn}{\end{eqnarray*}}
\newcommand{\beq}{\begin{eqnarray}}
\newcommand{\eeq}{\end{eqnarray}}

\newcommand{\rar}{\rightarrow}

\newcommand{\bc}{\begin{centre}}
\newcommand{\ec}{\end{centre}}

\newcommand{\ba}{\begin{array}}
\newcommand{\ea}{\end{array}}

\newcommand{\inp}[2]{\langle{#1},\,{#2} \rangle}

\begin{document}
\title[Weakly $\mathcal U(d)$-homogeneous]{Weakly $\mathcal U(d)$-homogeneous commuting tuple of bounded operators}
\author[S. Ghara]{Soumitra Ghara}
\author[S. Kumar]{Surjit Kumar}
\author[S. Trivedi]{Shailesh Trivedi}
\address{Department of Mathematics\\
Indian Institute of Technology  Kharagpur, Midnapore-721302, India}
   \email{soumitra@maths.iitkgp.ac.in\\ ghara90@gmail.com}
\address{Department of Mathematics \\  Indian Institute of Technology Madras, Chennai  600036, India}
\email{surjit@iitm.ac.in}
\address{Department of Mathematics\\
Birla Institute of Technology and Science, Pilani, Pilani Campus, Vidya Vihar, Pilani, Rajasthan 333031, India}
\email{shailesh.trivedi@pilani.bits-pilani.ac.in}

\thanks{The work of the first author is supported by INSPIRE Faculty Fellowship (DST/INSPIRE/04/2021/002555). The work of the second author was supported by the Anusandhan National Research Foundation (ANRF) through an IRG research grant (Ref. No. ANRF/IRG/2024/000432/MS). The work of the third author was partially supported by SERB-SRG (SRG/2023/000641-G) and OPERA (FR/SCM/03-Nov-2022/MATH)}

\date{}

\begin{abstract}
We introduce and study the weakly $\mathcal U(d)$-homogeneous commuting tuple of operators. We provide a sufficient condition under which a weakly $\mathcal U(d)$-homogeneous tuple is similar to a $\mathcal U(d)$-homogeneous tuple. Further, we focus our attention to multishifts and completely characterize weakly $\mathcal U(d)$-homogeneous multishifts. In particular, we show that a multishift is weakly $\mathcal U(d)$-homogeneous if and only if it similar to a $\mathcal U(d)$-homogeneous multishift. The results for multishifts are further refined for the class of spherically balanced multishifts.
\end{abstract}

\maketitle

\section{Introduction}

The study of homogeneous operator tuples has gained significant attention in recent years (see \cite{MS-1, MU, MK1}), and more recently, interest has expanded to include operator tuples that are homogeneous under the action of a compact group, see for instance \cite{CY, GKP, GKMP} and references therein. 
Recall that a bounded linear operator $T $ defined  on a complex separable Hilbert space with $\sigma(T) \subseteq \overline{\mathbb{D}}$ is said to be {\it homogeneous} if $T$ is unitarily equivalent to $\varphi(T)$ for all  $\varphi \in \mathrm{M\ddot ob},$   the group of biholomorphic  automorphism of the unit disc $\mathbb{D}.$
A complete classification of all homogeneous operators in the Cowen-Douglas class $B_1(\mathbb{D})$ was obtained by Misra \cite{Misra}.  In \cite{BM-1}, Bagchi and Misra provided a complete classification of homogeneous (scalar) weighted shifts.  All irreducible homogeneous operators in the class $B_2(\mathbb{D})$ were classified in  \cite{Wilkins}. Moreover,  Misra and Korányi \cite{MK} gave a complete description of homogeneous operators in the $B_n(\mathbb{D})$ capitalizing on the techniques from complex geometry and representation theory.

The notion of a weakly homogeneous operator  is obtained by replacing the unitary equivalence with similarity in the definition of a homogeneous operator \cite{C-M}. It is immediate that any operator similar to a homogeneous operator is weakly homogeneous. Moreover,  if $T$ is similar to a homogeneous operator then it must be M\"obius bounded, that is, $\sup_{\varphi\in \mathrm{M\ddot ob}}\|\varphi(T)\|$ is finite. In \cite{B-M}, a family of weakly homogeneous operators that fail to be M\"obius bounded was obtained, showing that not every weakly homogeneous operator is similar to a homogeneous operator.
In the same work, the authors asked whether every M\"obius bounded weakly homogeneous operator is similar to a homogeneous operator. This was later answered negatively by exhibiting a class of M\"obius bounded weakly homogeneous weighted shift operators which are not similar to any homogeneous operator \cite[Theorem 5.3]{Ghara}. This motivates us to study the operator tuples that are weakly homogeneous under the action of a compact group. In this note, the underlying group is the group $\mathcal U(d)$ of $d \times d$ unitary matrices. Before proceeding towards characterization of weakly $\mathcal U(d)$-homogeneous operator tuples, we set the notations as follows.

Let $\mathcal H$ be a complex separable Hilbert space and let $\mathcal B(\mathcal H)$ denote the algebra of bounded linear operators on $\mathcal H.$ For a set $X$ and a positive integer $d$, we denote by $X^{d}$ the $d$-fold Cartesian product of $X$. The symbols $\mathbb{N}$, $\mathbb{R}$, and $\mathbb{C}$ stands for the sets of non-negative integers, real numbers, and complex numbers, respectively. For  multi-index $\alpha = (\alpha_{1},\ldots,\alpha_{d}) \in \mathbb{N}^{d}$, we write
\[
|\alpha| := \sum_{j=1}^{d} \alpha_{j}, \qquad 
\alpha! := \prod_{j=1}^{d} \alpha_{j}!.
\]

The group $\mathcal{U}(d)$ acts on $\mathbb{C}^{d}$ via $z \mapsto u \cdot z$, where $u \cdot z$ denotes the usual matrix product. A $d$-tuple $T = (T_{1},\ldots,T_{d})$ of commuting bounded linear operators $T_1,\ldots,T_d$ on $H$ is said to be {\it $\mathcal{U}(d)$-homogeneous} (also called {\it spherical}) if, for every $u \in \mathcal U(d)$,  the tuple $u\cdot T$ is unitarily equivalent to $T.$
%Here, $u \cdot T$ denotes the natural action of $\mathcal{U}(d)$ on $T$.
A complete characterization of $\mathcal U(d)$-homogeneous (scalar valued) weighted multishifts is given in \cite[Theorem 2.1]{CY}. 
%Motivated by the notion of weakly homogeneous operators (refer to \cite{B-M} and \cite{Ghara}), we define {\it weakly $\mathcal U(d)$-homogeneous} operator tuple as follows : 
As in the case of weakly homogeneous operator, replacing unitary equivalence by similarity in the definition of $\mathcal U(d)$-homogeneous operator tuple gives rise to the notion of a {\it weakly $\mathcal U(d)$-homogeneous} operator tuple.

\begin{definition}
A commuting $d$-tuple $T=(T_1,\ldots,T_d)$ of bounded linear operators on a complex separable Hilbert space $\mathcal H$ is said to be weakly $\mathcal U(d)$-homogeneous if for every $u \in \mathcal U(d)$,  there exists a bounded invertible operator  $\Gamma(u)$ on $\mathcal H$ such that
\[
\Gamma(u)\, T_j = (u \cdot T)_j\, \Gamma(u), \qquad j = 1,\ldots,d.
\]
\end{definition}
Since similarity and unitary equivalence coincide for normal operators, it follows that  a  $d$-tuple of commuting normal operator is weakly U(d)-homogeneous  if and only if it is $\mathcal U(d)$-homogeneous. Note that every $\mathcal U(d)$-homogeneous operator tuple is weakly $\mathcal U(d)$-homogeneous. Moreover, any commuting tuple of bounded linear operators  similar to a $\mathcal U(d)$-homogeneous tuple is  necessarily weakly $\mathcal U(d)$-homogeneous. 
%Analogous to the case of weakly homogeneous operators in one variable, one may ask the following question:
%
%{\bf Question:} Does every weakly $\mathcal U(d)$-homogeneous operator tuple is similar to a $\mathcal U(d)$-homogeneous operator tuple?
%\textcolor{red}{I think we should not keep this question since we do not have an example yet. If we do keep it, discussing the difficulty level of this question compare to the weakly homogeneous case would be better.}

We briefly describe below the layout of this paper. 

In section $2,$ we begin with a sufficient condition for a weakly $\mathcal U(d)$-homogeneous operator tuple to be similar to a $\mathcal U(d)$-homogeneous tuple. In fact, we prove that every weakly $\mathcal U(d)$-homogeneous tuple along with the condition that the  intertwining operator $\Gamma(u)$ induces a strongly continuous representation $u \to \Gamma(u)$ of the group $\mathcal U(d),$ must be similar to a $\mathcal U(d)$-homogeneous tuple. Further, we show that the aforementioned sufficient condition holds for a weakly $\mathcal U(d)$-homogeneous tuple $\mathscr M_z=(\mathscr M_{z_1},\ldots, \mathscr M_{z_d})$ of multiplication operators by the coordinate functions $z_j$, $j = 1,\ldots, d$, on a reproducing kernel Hilbert space $\mathscr H(\kappa)$ of holomorphic functions on the open unit ball $\mathbb B_d$ of $\mathbb C^d$ with  $\mathscr H(\kappa) = \bigoplus_{n\in \mathbb N} \mathrm{Hom}(n)$, where $\mathrm{Hom}(n)$ denotes the space of homogeneous polynomials of degree $n$.

In Section $3$, we restrict our attention to the class of weighted multishifts. Recall that every weighted multishift is unitarily equivalent to a tuple of multiplication operators by the coordinate functions on a Hilbert space of formal power series $\mathcal H^2(\beta)$ (see \cite[Proposition 8]{JL}). Within this class, we first show that  weak $\mathcal U(d)$-homogeneity is equivalent to the boundedness of the composition operator $f \mapsto f\circ u$ on $\mathcal H^2(\beta)$ for all $u\in \mathcal U(d).$ Then using this equivalence, we show that the sufficient condition obtained in Section 2 is automatic for this class too. That is,  every weakly $\mathcal U(d)$-homogeneous weighted multishift is similar to a $\mathcal U(d)$-homogeneous weighted multishift.
Furthermore, a criterion for weak $\mathcal U(d)$-homogeneity of the weighted multishift is provided in terms of its moments.
 
Section $4$ is devoted to a complete characterization of balanced weighted multishifts (see \cite[Definition~1.5]{C-K}) that are weakly $\mathcal U(d)$-homogeneous. In particular, we show that if the multiplication $d$-tuple $\mathscr M_z$ on $\mathcal H^2(\mu)$ is weakly $\mathcal U(d)$-homogeneous, then it must be similar to the Szegő shift. Here $\mathcal H^2(\mu)$ denotes the closure of polynomials in $L^2(\partial\mathbb B_d,\mu)$, where $\mu$ is a finite positive Reinhardt (Borel) measure supported on the unit sphere $\partial\mathbb B_d$.

%\textcolor{red}{It would be a good idea to refer to the corresponding results, that is, to specify which theorem or proposition we are discussing in the brief details of the paper given above.}

\section{A sufficient condition} 

In this section, we provide a sufficient condition under which a weakly $\mathcal U(d)$-homogeneous operator tuple is similar to a $\mathcal U(d)$-homogeneous tuple. Further, we show that this condition is automatically satisfied by a weakly $\mathcal U(d)$-homogeneous $d$-tuple $\mathscr M_z$ on the class of reproducing kernel Hilbert spaces $\mathscr H(\kappa)$ of holomorphic functions on the open unit ball $\mathbb B_d$ of $\mathbb C^d$ with  $\mathscr H(\kappa) = \bigoplus_{n\in \mathbb N} \mathrm{Hom}(n)$.

The techniques used in the proof of the following proposition are standard in representation theory, we include the details for the sake of completeness.

\begin{proposition}\label{gen-thm}
Let $T = (T_1, \ldots, T_d)$ be a commuting tuple of bounded linear operators on a complex separable Hilbert space $\mathcal H$. Suppose that $T$ is weakly $\mathcal U(d)$-homogeneous and $\Gamma(u)$ be a bounded invertible operator on $H$ such that
\[\Gamma(u)T_j = (u\cdot T)_j \Gamma(u ),\ \text{ for all } j=1, \ldots, d \ \text{ and }\ u \in \mathcal U(d).\]
If $u \mapsto \Gamma(u)$ is a strongly continuous representation of\, $\mathcal U(d)$, then $T$ is similar to a $\mathcal U(d)$-homogeneous tuple on $\mathcal H$. 
\end{proposition}

\begin{proof}
Suppose that the map $\pi : \mathcal U(d) \to \mathcal B(\mathcal H)$ defined by $\pi(u) = \Gamma(u)$ is a strongly continuous representation of $\mathcal U(d)$. Consequently, the map $u \mapsto \|\pi(u)x\|$ is a continuous map on $\mathcal U(d)$ for each $x \in \mathcal H$. Since $\mathcal U(d)$ is compact, we get that 
\[\sup_{u \in \mathcal U(d)} \|\pi(u)x\| \leqslant M(x) < \infty.\]
By uniform boundedness principle, we obtain that 
\[c := \sup_{u \in \mathcal U(d)} \|\pi(u)\| < \infty.\]
Thus, $\pi$ is a strongly continuous uniformly bounded representation of the compact group $\mathcal U(d)$. Let $du$ denote the Haar measure of $\mathcal U(d)$. Define an inner product $\ll \cdot,\cdot\gg$ on $\mathcal H$ as 
\beqn
\ll x,y\gg\ := \int_{\mathcal U(d)} \inp{\pi(u) x}{\pi(u) y} du, \quad x, y \in \mathcal H.
\eeqn
Let $\|\cdot\|_1$ denote the norm induced by $\ll \cdot,\cdot\gg$. Then it is easy to see that 
\beqn
\|x\|_1 \leqslant c \|x\| \ \text{ and }\ \|x\| \leqslant c \|x\|_1 \ \text{ for all }\ x \in \mathcal H.
\eeqn
Thus, $(\mathcal H, \|\cdot\|_1)$ is a Hilbert space. Routine verification shows that $\pi(u)$ is a unitary on $(\mathcal H, \|\cdot\|_1)$ for all $u \in \mathcal U(d)$. It follows from the Riesz representation theorem that $(\mathcal H, \|\cdot\|_1)$ is isometric to $\mathcal H$. Let $S : (\mathcal H, \|\cdot\|_1) \to \mathcal H$ be a unitary map. Then $\rho(u) := S \pi(u) S^{-1}$, $u \in \mathcal U(d)$, defines a strongly continuous unitary representation of $\mathcal U(d)$. Viewing $S$ as an operator on $\mathcal H$, we obtain that $S$ is a bounded invertible operator on $\mathcal H$ and $\pi(u) = S^{-1} \rho(u) S$ for all $u \in \mathcal U(d)$. Note that 
\beqn
\rho(u) S T_j S^{-1} = S \pi(u) T_j S^{-1} = S (u \cdot T)_j \pi(u) S^{-1} 
= S (u \cdot T)_j S^{-1} \rho(u).
\eeqn
The above equality shows that the $d$-tuple $S T S^{-1} := (S T_1 S^{-1}, \ldots, S T_d S^{-1})$ is $\mathcal U(d)$-homogeneous. This completes the proof.  
\end{proof}

\begin{remark}
The conclusion of the preceding proposition also follows from a more general fact, namely, Dixmier's theorem \cite{D}. In fact, since $\mathcal U(d)$ is amenable (see \cite[Page 6]{P}), by Dixmier's theorem \cite{D} there exists a bounded invertible operator $S$ on $\mathcal H$ such that $\pi(u) = S^{-1} \rho(u) S$ for all $u \in \mathcal U(d)$, where $\rho : \mathcal U(d) \to \mathcal B(\mathcal H)$ is a strongly continuous unitary representation. Consequently, $S T S^{-1}$ is $\mathcal U(d)$-homogeneous.
\end{remark}

Let $p(z) = \sum_{|\alpha| \leqslant n} a_\alpha z^\alpha$, $a_\alpha \in \mathbb C$, $n \in \mathbb N$, be a polynomial in the variable $z = (z_1, \ldots, z_d)$. We say that $p$ is a homogeneous polynomial of degree $n$ if $p(tz) = t^n p(z)$, where $tz = (tz_1, \ldots, tz_d)$, $t \in \mathbb R$. Let $\text{Hom}(n)$ denote the vector space of all homogeneous polynomials of degree $n$. Let $\mathscr H(\kappa)$ be the reproducing kernel Hilbert space of holomorphic functions on the open unit ball $\mathbb B_d$ of $\mathbb C^d$ such that $\mathscr H(\kappa) = \bigoplus_{n\in \mathbb N} \mathrm{Hom}(n)$. For $u \in \mathcal U(d)$, we define the linear transformation $C_u$ in $\mathscr H(\kappa)$ as 
\beq\label{Cu-eq}
C_u f (z) = f(u \cdot z),\quad f \in \mathcal D(C_u).
\eeq
Note that
\beqn\label{Cu-Hom-n}
C_u(\text{Hom}(n)) = \text{Hom}(n), \quad n \in \mathbb N.
\eeqn
Thus, $C_u$ is a densely defined linear map which may not be bounded on $\mathscr H(\kappa)$. If $C_u$ is bounded on $\mathscr H(\kappa)$, then we refer to $C_u$ as {\it composition operator} induced by $u$.

We now show that the sufficient condition obtained in the Proposition \ref{gen-thm} is automatically satisfied by weakly $\mathcal U(d)$-homogeneous $\mathscr M_z$ on $\mathscr H(\kappa)$ which admit the orthogonal decomposition $\mathscr H(\kappa) = \bigoplus_{n\in \mathbb N} \mathrm{Hom}(n)$. We proceed with the following elementary lemma whose proof is included for the sake of completeness.

\begin{lemma}\label{Cu-strongly-rkhs}
Let $\mathscr H(\kappa)$ be a reproducing kernel Hilbert space of complex-valued holomorphic functions defined on the open unit ball $\mathbb B_d$ of $\mathbb C^d$. Suppose that $\mathscr H(\kappa) = \bigoplus_{n\in \mathbb N} \mathrm{Hom}(n)$. If $C_u$ is a bounded operator on $\mathscr H(\kappa)$ for each $u \in \mathcal U(d)$, then the map $\pi : \mathcal U(d) \to \mathcal B\big(\mathscr H(\kappa)) \big)$ given by $\pi(u) = C_{u^{-1}}$, $u \in \mathcal U(d)$, is a strongly continuous representation of $\mathcal U(d)$ .
\end{lemma}

\begin{proof}
That $\pi$ is a representation is a routine verification. Fix $f \in \mathscr H(\kappa)$. Then 
\[f(z) = \sum_{n \in \mathbb N} p_n(z),\quad p_n \in \text{Hom}(n).\]
Let $\varepsilon > 0$. Since $C_{u^{-1}}$, $C_{v^{-1}}$ are bounded operators on $\mathscr H(\kappa)$ for $u,v \in \mathcal U(d)$, and $C_u(\text{Hom}(n)) = \text{Hom}(n)$ for all $n \in \mathbb N$, it follows that there exists a positive integer $N$ such that
\beqn
\|C_{u^{-1}} f - C_{v^{-1}} f\|^2_{\mathscr H(\kappa)} &=& \sum_{n=0}^N \|p_n(u^{-1} \cdot z) - p_n(v^{-1} \cdot z)\|^2_{\mathscr H(\kappa)}\\ 
&+& \sum_{n = N+1}^\infty \|p_n(u^{-1} \cdot z) - p_n(v^{-1} \cdot z)\|^2_{\mathscr H(\kappa)}\\
&<& \sum_{n=0}^N \|p_n(u^{-1} \cdot z) - p_n(v^{-1} \cdot z)\|^2_{\mathscr H(\kappa)} + \varepsilon. 
\eeqn
Note that $p(z) := \sum_{n=0}^N p_n(z)$ is a polynomial and $u \mapsto u^{-1} \cdot z$ is a continuous function from $\mathcal U(d)$ to $\mathbb C^d$. Therefore, there exists a $\delta > 0$ such that $\|p(u^{-1} \cdot z) - p(v^{-1} \cdot z)\|^2_{\mathscr H(\kappa)} < \varepsilon$ whenever $\|u-v\|<\delta$. Thus, from above, we get
\[\|C_{u^{-1}} f - C_{v^{-1}} f\|_{\mathscr H(\kappa)} < \sqrt{2\varepsilon}\ \text{ whenever }\ \|u-v\|<\delta.\]
This completes the proof.
\end{proof}

Here is the main result of this section.

\begin{theorem}\label{rkhs-Ud-homo}
Let $\mathscr H(\kappa)$ be a reproducing kernel Hilbert space of complex-valued holomorphic functions defined on the open unit ball $\mathbb B_d$ of $\mathbb C^d$. Suppose that $\mathscr H(\kappa) = \bigoplus_{n\in \mathbb N} \mathrm{Hom}(n)$. Then $\mathscr M_z$ on $\mathscr H(\kappa)$ is weakly $\mathcal U(d)$-homogeneous if and only if $\mathscr M_z$ is similar to a $\mathcal U(d)$-homogeneous tuple.
\end{theorem}

\begin{proof}
For $u \in \mathcal U(d)$, define
\beqn\label{Ku-inverse}
\kappa_{u}(z, w) := \kappa(u\cdot z, u \cdot w), \quad z, w \in \mathbb B_d.
\eeqn
Then it is easy to see that for all $u \in \mathcal U(d)$, $\mathscr H(\kappa_u)$ is also a reproducing kernel Hilbert space of complex-valued holomorphic functions $\mathbb B_d$ with $\mathscr H(\kappa_u) = \bigoplus_{n\in \mathbb N} \mathrm{Hom}(n)$. Note that the map $T_{u^{-1}} : \mathscr H(\kappa) \rar \mathscr H(\kappa_{u^{-1}})$, defined as $T_{u^{-1}} f(z) = f(u^{-1} \cdot z)$ for all $f \in \mathscr H(\kappa)$, is a unitary, and 
\beqn
	(T_{u^{-1}} (u \cdot \mathscr M_z)_j f)(z) = (T_{u^{-1}} (u \cdot z)_j f)(z) = z_j f(u^{-1} \cdot z) =  (\mathscr M_{z_j} T_{u^{-1}} f) (z).
	\eeqn
Now suppose that $\mathscr M_z$ on $\mathscr H(\kappa)$ is weakly $\mathcal U(d)$-homogeneous and for $u \in \mathcal U(d)$, let $\Gamma(u)$ be a bounded invertible operator on $\mathscr H(\kappa)$ such that
\beqn
\Gamma(u)\mathscr M_{z_j} = (u \cdot \mathscr M_z)_j \Gamma(u)\ \text{ for all }\ j=1, \ldots, d.
\eeqn
Then the map $X_u := T_{u^{-1}}\Gamma(u) : \mathscr H(\kappa) \to  \mathscr H(\kappa_{u^{-1}})$ is a bounded invertible operator which intertwines $\mathscr M_z$ on $\mathscr H(\kappa)$ and on $\mathscr H(\kappa_{u^{-1}})$. Therefore, by \cite[Theorem 5.2]{GKP}, there exist positive constants $C_1(u)$ and $C_2(u)$ such that
\beqn
C_1(u) \kappa(z,w) \leqslant \kappa_{u^{-1}}(z,w) \leqslant C_2(u) \kappa(z,w), \quad u \in \mathcal U(d) \text{ and } z,w \in \mathbb B_d.
\eeqn 
It now follows from \cite[Theorem 5.10]{PR} that $C_{u^{-1}} : \mathscr H(\kappa) \to \mathscr H(\kappa)$ is a bounded invertible operator. It is easy to see that 
\beqn
C_{u^{-1}} \mathscr M_{z_j} = (u^{-1} \cdot \mathscr M_z)_j C_{u^{-1}}\ \text{ for all }\ j \in \{1, \ldots, d\}.
\eeqn
Thus, by Lemma \ref{Cu-strongly-rkhs} and Proposition	\ref{gen-thm}, we conclude that $\mathscr M_z$ is similar to a $\mathcal U(d)$-homogeneous tuple. The converse is obvious.
\end{proof}

\section{Weakly $\mathcal U(d)$-homogeneous multishifts}

In this section, we show that the sufficient condition obtained in the Proposition \ref{gen-thm} is satisfied by the weakly $\mathcal U(d)$-homogeneous  (weighted) multishifts. The main result of this section states that a weakly $\mathcal U(d)$-homogeneous multishift is similar to a $\mathcal U(d)$-homogeneous multishift. It is well known that a multishift on $\ell^2(\mathbb N^d)$ can be looked upon as $\mathscr M_z$ on a Hilbert space $\mathcal H^2(\beta)$ of formal power series \cite{JL}. In view of this equivalence, we work in $\mathcal H^2(\beta)$ setup. The reader is referred to \cite{GKT, GKT1, GKT2} for a detailed study of multishifts.

Recall that for a multisequence $\beta = \{\beta_\alpha : \alpha \in \mathbb N^d\}$ of positive real  numbers, the Hilbert space $\mathcal H^2(\beta)$ of formal power series is defined as 
\beqn
\mathcal H^2(\beta) := \bigg\{ \sum_{\alpha \in \mathbb N^d} a_\alpha z^\alpha : \sum_{\alpha \in \mathbb N^d} |a_\alpha|^2 \beta_{\alpha}^2 < \infty,\ a_\alpha \in \mathbb C\bigg\}
\eeqn
equipped with the following inner product: 
\beqn
\inp{f}{g} := \sum_{\alpha \in \mathbb N^d} a_\alpha \overline{b}_\alpha \beta_{\alpha}^2, \quad f(z) = \sum_{\alpha \in \mathbb N^d} a_\alpha z^\alpha,\ g(z) = \sum_{\alpha \in \mathbb N^d} b_\alpha z^\alpha \in \mathcal H^2(\beta).
\eeqn

For a unitary matrix $u=(u_{ij})_{i,j=1}^d$, we define 
\beqn\label{H2ubeta}
\mathcal H_u^2(\beta) := \bigg\{ \sum_{\alpha \in \mathbb N^d} a_\alpha (u\cdot z)^\alpha : \sum_{\alpha \in \mathbb N^d} |a_\alpha|^2 \beta_{\alpha}^2 < \infty,\ a_\alpha \in \mathbb C\bigg\},
\eeqn
where for $z=(z_1, \ldots, z_d) \in \mathbb C^d$, 
$$(u\cdot z) := \Big(\sum_{k=1}^d u_{1k} z_k, \ldots,  \sum_{k=1}^d u_{dk} z_k\Big).$$ 
It is easily seen that $\mathcal H_u^2(\beta)$ is also a Hilbert space of formal power series in $d$-variables $z_1, \ldots, z_d$. Moreover, $\mathcal H_u^2(\beta)$ is same as $\mathcal H^2(\beta)$ as a set. Further, for $\mathscr M_z = (\mathscr M_{z_1}, \ldots, \mathscr M_{z_d})$, we set
$$(u\cdot \mathscr M_z) := \Big(\sum_{k=1}^d u_{1k} \mathscr M_{z_k}, \ldots,  \sum_{k=1}^d u_{dk} \mathscr M_{z_k}\Big).$$

The following lemma describes the intertwiner of $\mathscr M_z$ on $\mathcal H^2(\beta)$ and on $\mathcal H_u^2(\beta)$.

\begin{lemma}\label{multiplier}
Let $u \in \mathcal U(d)$. Suppose that $\mathscr M_z$ is bounded on $\mathcal H^2(\beta)$ and on $\mathcal H_u^2(\beta)$. If $X : \mathcal H^2(\beta) \to \mathcal H_u^2(\beta)$ is a bounded linear map such that $X \mathscr M_{z_j} = \mathscr M_{z_j} X$, $1 \leqslant j \leqslant d$, then $X = \mathscr M_\phi$, where 
$$\phi(z) = \sum_{\alpha \in \mathbb N^d} a_\alpha (u\cdot z)^\alpha,\quad a_\alpha \in \mathbb C.$$ 
Further, if $X$ is invertible, then $a_0$ is non-zero.
\end{lemma}

\begin{proof}
Suppose that $X : \mathcal H^2(\beta) \to \mathcal H_u^2(\beta)$ is a bounded linear map such that $X \mathscr M_{z_j} = \mathscr M_{z_j} X$ for all $j = 1, \ldots, d$. Let
\[X(1) = \sum_{\alpha \in \mathbb N^d} a_\alpha (u\cdot z)^\alpha \in \mathcal H_u^2(\beta).\]
Since $X$ intertwines $\mathscr M_{z_j}$ for all $j = 1, \ldots, d$, we get
\[X(z^\alpha) = X\mathscr M_z^\alpha (1) = \mathscr M_z^\alpha X(1) = z^\alpha X(1) \text{ for all } \alpha \in \mathbb N^d.\]
Let $f \in \mathcal H^2(\beta)$ and $f(z) = \sum_{\alpha \in \mathbb N^d} f(\alpha) z^\alpha$. Then 
\beqn
Xf(z) &=& \sum_{\alpha \in \mathbb N^d} X\big(f(\alpha) z^\alpha\big) = \sum_{\alpha \in \mathbb N^d} f(\alpha) X(z^\alpha) = \sum_{\alpha \in \mathbb N^d} f(\alpha) z^\alpha X(1)\\ 
&=& \phi(z) f(z),
\eeqn
where $\phi(z) = \sum_{\alpha \in \mathbb N^d} a_\alpha (u\cdot z)^\alpha$. 

Now suppose that $X$ is invertible. Then following the arguments of the preceding paragraph, we can show that $X^{-1} = \mathscr M_\psi$, where 
\[\psi(z) = \sum_{\alpha \in \mathbb N^d} b_\alpha z^\alpha,\quad b_\alpha \in \mathbb C.\] 
Let $[1]$ denote the subspace of $\mathcal H^2(\beta)$ spanned by the constant polynomial $1$. Then
\beqn
1 = I_{[1]}(1) = P_{[1]} X^{-1} X|_{[1]}(1) = P_{[1]} \mathscr M_\psi \mathscr M_\phi(1) = P_{[1]} \psi(z)\phi(z)= a_0 b_0.
\eeqn
Thus, $a_0$ is non-zero. This completes the proof.
\end{proof}

As in the previous section, for $u \in \mathcal U(d)$, we consider the linear transformation $C_u$ in $\mathcal H^2(\beta)$ given by \eqref{Cu-eq}, that is, 
\beqn
C_u f (z) = f(u \cdot z),\quad f \in \mathcal D(C_u).
\eeqn
Note that $C_u$ is a densely defined linear map which may not be bounded on $\mathcal H^2(\beta)$. If $C_u$ is bounded on $\mathcal H^2(\beta)$, then we refer to $C_u$ as composition operator induced by $u$. Later we show that $\mathscr M_z$ being weakly $\mathcal U(d)$-homogeneous is equivalent to $C_u$ being bounded on $\mathcal H^2(\beta)$ for all $u \in \mathcal U(d)$. The following lemma is needed to this end, which describes the intertwiner of $\mathscr M_z$ and $u \cdot \mathscr M_z$ on $\mathcal H^2(\beta)$.

\begin{lemma}\label{lem-Ud}
Suppose that $\mathscr M_z$ is bounded on $\mathcal H^2(\beta)$. For $u \in \mathcal U(d),$ let $\Gamma(u): \mathcal H^2(\beta) \to \mathcal H^2(\beta)$ be a bounded linear map such that $\Gamma(u)\mathscr M_{z_j} = (u \cdot \mathscr M_z)_j \Gamma(u)$ for all $j = 1, \ldots, d$. Then $\Gamma(u) = \mathscr M_\phi C_u$, where $\phi(z) = \sum_{\alpha \in \mathbb N^d} a_\alpha z^\alpha$, $a_\alpha \in \mathbb C$.
\end{lemma}

\begin{proof}
Fix $u \in \mathcal U(d)$. Then the map $T_{u^{-1}} : \mathcal H^2(\beta) \rar \mathcal H_{u^{-1}}^2(\beta)$, defined as $T_{u^{-1}} f(z) = f(u^{-1} \cdot z)$ for all $f \in \mathcal H^2(\beta)$, is a unitary. Further, for all $f \in \mathcal H^2(\beta)$ and $j \in \{1, \ldots, d\}$, we have
	\beqn
	(T_{u^{-1}} (u \cdot \mathscr M_z)_j f)(z) = (T_{u^{-1}} (u \cdot z)_j f)(z) = z_j f(u^{-1} \cdot z) =  (\mathscr M_{z_j} T_{u^{-1}} f) (z).
	\eeqn
Thus, $u \cdot \mathscr M_z$ on $\mathcal H^2(\beta)$ is unitarily equivalent to $\mathscr M_z$ on $\mathcal H_{u^{-1}}^2(\beta)$. Let $\Gamma(u): \mathcal H^2(\beta) \to \mathcal H^2(\beta)$ be a bounded invertible operator such that $\Gamma(u)\mathscr M_{z_j} = (u \cdot \mathscr M_z)_j\Gamma(u)$ for all $j = 1, \ldots, d$. This yields
	\beqn
	T_{u^{-1}}\Gamma(u) \mathscr M_{z_j} = T_{u^{-1}}(u \cdot \mathscr M_z)_j \Gamma(u) = \mathscr M_{z_j}	T_{u^{-1}}\Gamma(u), \quad j = 1, \ldots, d.
	\eeqn
	Note that $X_u := T_{u^{-1}}\Gamma(u) : \mathcal H^2(\beta) \to  \mathcal H_{u^{-1}}^2(\beta)$ is a bounded invertible operator which intertwines $\mathscr M_z$ on $\mathcal H^2(\beta)$ and on $\mathcal H_{u^{-1}}^2(\beta)$. Therefore, by Lemma \ref{multiplier}, $X_u f(z)=\sum_{\alpha \in \mathbb N^d} a_\alpha (u^{-1} \cdot z)^\alpha f(z),$ where $a_\alpha \in \mathbb C.$
	This shows that \beqn  \Gamma(u)f(z)=T_{u} \sum_{\alpha \in \mathbb N^d} a_\alpha (u^{-1} \cdot z)^\alpha f(z) =\phi(z)f(u \cdot z),  \eeqn where $\phi(z) = \sum_{\alpha \in \mathbb N^d} a_\alpha z^\alpha$, $a_\alpha \in \mathbb C$. Thus, $\Gamma(u) = \mathscr M_\phi C_u$.
\end{proof}

The following proposition shows that $\mathscr M_z$ being weakly $\mathcal U(d)$-homogeneous on $\mathcal H^2(\beta)$ is equivalent to the composition transformation $C_u$ being bounded on $\mathcal H^2(\beta)$ for all $u \in \mathcal U(d)$.
	
\begin{proposition}\label{prop-2.13}
Let $\mathscr M_z$ be bounded on $\mathcal H^2(\beta)$. Then $\mathscr M_z$ is weakly $\mathcal U(d)$-homogeneous if and only if $C_u$ is a bounded operator on $\mathcal H^2(\beta)$ for all $u \in \mathcal U(d)$. 
\end{proposition}

\begin{proof}
For $f \in \mathcal H^2(\beta)$ and $t \in \mathbb T^d$, define $f_t(z) := f(t\cdot z)$, where $t\cdot z := (t_1 z_1, \ldots, t_d z_d)$. Then it is easy to see that $f_t \in \mathcal H^2(\beta)$ and $\|f_t\|_{\mathcal H^2(\beta)} = \|f\|_{\mathcal H^2(\beta)}$ for all $f \in \mathcal H^2(\beta)$ and $t \in \mathbb T^d$. Suppose that $\mathscr M_z$ is weakly $\mathcal U(d)$-homogeneous. For $u \in \mathcal U(d)$, let $\Gamma(u) : \mathcal H^2(\beta) \to \mathcal H^2(\beta)$ be an invertible operator such that
\beqn
\Gamma(u)\mathscr M_{z_j} = (u \cdot \mathscr M_z)_j \Gamma(u), \quad j = 1, \ldots, d.
\eeqn
Then it follows from Lemma \ref{lem-Ud} that $\Gamma(u) = \mathscr M_\phi C_u$. Therefore, for $f \in \mathcal H^2(\beta)$ and $t \in \mathbb T^d$, we have
\beqn
\big(\Gamma(u) f_t\big)_{\bar t}(z) = \phi(\bar t \cdot z) f(u \cdot z).
\eeqn 
Thus, we get
\beqn
\int_{\mathbb T^2} \big(\Gamma(u) f_t\big)_{\bar t}(z)\, dm(t) = \phi(0) f(u \cdot z) \text{ for all }\ u \in \mathcal U(d), 
\eeqn
where $m(t)$ is the normalized Lebesgue measure on $\mathbb T^d$. This, in turn, yields that
\beqn
|\phi(0) |^2 \|f(u \cdot z)\|^2 &=& \left\| \int_{\mathbb T^d} \big(\Gamma(u) f_t\big)_{\bar t}(z)\, dm(t) \right\|^2 \leqslant \int_{\mathbb T^d} \|\big(\Gamma(u) f_t\big)_{\bar t}\|^2 dm(t)\\
&=& \int_{\mathbb T^d} \|\Gamma(u) f_t\|^2 dm(t) \leqslant \int_{\mathbb T^d} \|\Gamma(u)\|^2 \|f_t\|^2 dm(t)\\
&=& \|\Gamma(u)\|^2 \|f\|^2\ \text{ for all } f \in \mathcal H^2(\beta) \text{ and } u \in \mathcal U(d). 
\eeqn
Note that $\phi(0)$ is non-zero by Lemma \ref{multiplier}. Hence, we get
\beqn
\|f(u \cdot z)\| \leqslant \frac{\|\Gamma(u)\|}{|\phi(0)|}\, \|f\| \text{ for all } f \in \mathcal H^2(\beta) \text{ and } u \in \mathcal U(d).
\eeqn
Thus, $C_u$ is bounded on $\mathcal H^2(\beta)$ for all $u \in \mathcal U(d)$.

Conversely, suppose that $C_u$ is bounded on $\mathcal H^2(\beta)$ for all $u \in \mathcal U(d)$. Then it is easy to see that $C_u$ is a bounded invertible operator on $\mathcal H^2(\beta)$ for all $u \in \mathcal U(d)$. Observe that 
\beq\label{Cu-intertwine}
C_u \mathscr M_{z_j} f = C_u (z_j f) = (u\cdot z)_j f(u \cdot z) = (u \cdot \mathscr M_z)_j C_u f 
\eeq
for all $f \in \mathcal H^2(\beta)$, $j \in \{1, \ldots, d\}$, and $u \in \mathcal U(d)$. Thus, $\mathscr M_z$ is weakly $\mathcal U(d)$-homogeneous. This completes the proof.
\end{proof}

The proof of the following lemma is similar to that of Lemma \ref{Cu-strongly-rkhs}, and hence, it is stated here without proof.

\begin{lemma}\label{Cu-strongly}
Let $C_u$ be a bounded operator on $\mathcal H^2(\beta)$ for each $u \in \mathcal U(d)$. Then the map $\pi : \mathcal U(d) \to \mathcal B\big(\mathcal H^2(\beta)\big)$ given by $\pi(u) = C_{u^{-1}}$, $u \in \mathcal U(d)$, is a strongly continuous representation of $\mathcal U(d)$ into $\mathcal B\big(\mathcal H^2(\beta)\big)$.
\end{lemma}

Now we present the main result of this section.

\begin{theorem}\label{Cu-bdd}
Let $\mathscr M_z$ be bounded on $\mathcal H^2(\beta)$. Then the following statements are equivalent: 
\vspace{.3cm}
\begin{itemize}
\item[$(i)$] $\mathscr M_z$ on $\mathcal H^2(\beta)$ is weakly $\mathcal U(d)$-homogeneous.\vspace{.2cm}
\item[$(ii)$] $\mathscr M_z$ on $\mathcal H^2(\beta)$ is similar to a $\mathcal U(d)$-homogeneous tuple $\mathscr M_z$ on some $\mathcal H^2(\tilde \beta)$.\\
\Item[$(iii)$] \vspace{-.3cm}\beqn\label{eq-1}
\hspace{-1.5cm}\sup_{n \in \mathbb N}\left\{\max_{|\alpha|=n} \frac{\|z^\alpha\|_{\mathcal H^2(\beta)}}{\sqrt{\alpha!}} \max_{|\delta|=n} \frac{\sqrt{\delta!}}{\|z^\delta\|_{\mathcal H^2(\beta)}}\right\} < \infty.
\eeqn
\item[$(iv)$] The family $\{C_u : u \in \mathcal U(d)\}$ of composition operators on $\mathcal H^2(\beta)$ is uniformly bounded.
\end{itemize}
\end{theorem}

\begin{proof}$(i) \implies (ii)$: Suppose that $\mathscr M_z$ is weakly $\mathcal U(d)$-homogeneous. Then it follows from Proposition \ref{prop-2.13} and Lemma \ref{Cu-strongly} that the map $\pi : \mathcal U(d) \to \mathcal B\big(\mathcal H^2(\beta)\big)$ given by $\pi(u) = C_{u^{-1}}$, $u \in \mathcal U(d)$, is a strongly continuous representation of $\mathcal U(d)$ into $\mathcal B\big(\mathcal H^2(\beta)\big)$. Hence, by Proposition \ref{gen-thm}, there exists a bounded invertible operator $S$ on $\mathcal H^2(\beta)$ such that the $d$-tuple $S \mathscr M_z S^{-1} := (S \mathscr M_{z_1} S^{-1}, \ldots, S \mathscr M_{z_d} S^{-1})$ is $\mathcal U(d)$-homogeneous. It is not difficult to verify that $\dim \ker S^{*-1} \mathscr M_z^* S^* = 1$ and $\ker S^{*-1} \mathscr M_z^* S^*$ is a cyclic subspace for $S \mathscr M_z S^{-1}$. Therefore, by \cite[Theorem 2.5]{CY}, we get that $S \mathscr M_z S^{-1}$ is unitarily equivalent to $\mathscr M_z$ on some $\mathcal H^2(\tilde\beta)$. Consequently, we obtain that $\mathscr M_z$ on $\mathcal H^2(\beta)$ is similar to a $\mathcal U(d)$-homogeneous tuple $\mathscr M_z$ on some $\mathcal H^2(\tilde \beta)$, which proves $(ii)$.\\  

$(ii) \implies (iii)$: Suppose that $\mathscr M_z$ on $\mathcal H^2(\beta)$ is similar to a $\mathcal U(d)$-homogeneous tuple $\mathscr M_z$ on some $\mathcal H^2(\tilde\beta)$. It follows from \cite[Theorem 2.1]{CY} that there exists a sequence $(a_n)_{n \geqslant 0}$ of positive numbers such that  
\beqn\label{eq-3}
\|z^\alpha\|_{\mathcal H^2(\tilde\beta)} = a_{|\alpha|}\frac{\sqrt{(d-1)!\alpha!}}{\sqrt{(d-1+|\alpha|)!}}\ \text{ for all }\ \alpha \in \mathbb N^d.
\eeqn
This gives that
\beq\label{eq-3.5}
\max_{|\alpha|=n} \frac{\|z^\alpha\|_{\mathcal H^2(\tilde\beta)}}{\sqrt{\alpha!}} \max_{|\alpha|=n} \frac{\sqrt{\alpha!}}{\|z^\alpha\|_{\mathcal H^2(\tilde\beta)}} &=& a_{n}\frac{\sqrt{(d-1)!}}{\sqrt{(d-1+n)!}} \frac{\sqrt{(d-1+n)!}}{a_{n}\sqrt{(d-1)!}}\notag\\
&=& 1. 
\eeq
Also, since $\mathscr M_z$ on $\mathcal H^2(\beta)$ is similar to $\mathscr M_z$ on $\mathcal H^2(\tilde\beta)$, it follows from \cite[Lemma 2.2]{K} that there exist positive constants $m_1$ and $m_2$ such that 
\beq\label{eq-4}
m_1 \leqslant \frac{\|z^\alpha\|_{\mathcal H^2(\beta)}}{\|z^\alpha\|_{\mathcal H^2(\tilde\beta)}} \leqslant m_2\ \text{ for all }\ \alpha \in \mathbb N^d.
\eeq
Consequently, we have
\beqn
m_1 \frac{\|z^\alpha\|_{\mathcal H^2(\tilde\beta)}}{\sqrt{\alpha!}} \leqslant \frac{\|z^\alpha\|_{\mathcal H^2(\beta)}}{\sqrt{\alpha!}} \leqslant m_2 \frac{\|z^\alpha\|_{\mathcal H^2(\tilde\beta)}}{\sqrt{\alpha!}}
\eeqn
for all $\alpha \in \mathbb N^d$, which in turn implies that for all $n \in \mathbb N$,
\beq\label{eq-5}
m_1 \max_{|\alpha|=n}\frac{\|z^\alpha\|_{\mathcal H^2(\tilde\beta)}}{\sqrt{\alpha!}} \leqslant \max_{|\alpha|=n} \frac{\|z^\alpha\|_{\mathcal H^2(\beta)}}{\sqrt{\alpha!}} \leqslant m_2 \max_{|\alpha|=n} \frac{\|z^\alpha\|_{\mathcal H^2(\tilde\beta)}}{\sqrt{\alpha!}}.
\eeq
 Now interchanging the role of $\tilde\beta$ and $\beta$ in \eqref{eq-4} and proceeding similarly, we obtain that
\beq\label{eq-6}
c_1 \max_{|\alpha|=n}\frac{\sqrt{\alpha!}}{\|z^\alpha\|_{\mathcal H^2(\tilde\beta)}} \leqslant \max_{|\alpha|=n} \frac{\sqrt{\alpha!}}{\|z^\alpha\|_{\mathcal H^2(\beta)}} \leqslant c_2 \max_{|\alpha|=n} \frac{\sqrt{\alpha!}}{\|z^\alpha\|_{\mathcal H^2(\tilde\beta)}}
\eeq
for some positive constants $c_1,\, c_2$ and for all $n \in \mathbb N$. By \eqref{eq-5} and \eqref{eq-6} together with \eqref{eq-3.5}, we conclude that 
\beqn
\sup_{n \in \mathbb N}\left\{\max_{|\alpha|=n} \frac{\|z^\alpha\|_{\mathcal H^2(\beta)}}{\sqrt{\alpha!}} \max_{|\delta|=n} \frac{\sqrt{\delta!}}{\|z^\delta\|_{\mathcal H^2(\beta)}}\right\} < \infty.
\eeqn

$(iii) \implies (iv)$: Assume that $(iii)$ holds. Fix $u \in \mathcal U(d)$. We first show that $C_u$ is bounded on $\mathcal H^2(\beta)$. Note that 
\beqn
\mathcal H^2(\beta) = \bigoplus_{n\in \mathbb N} \text{Hom}(n)
\eeqn
and for all $u \in \mathcal U(d)$,
\beqn
C_u(\text{Hom}(n)) = \text{Hom}(n), \quad n \in \mathbb N.
\eeqn
In view of this, it is enough to show that $\sup_{n \in \mathbb N} \|C_u|_{\text{Hom}(n)}\| < \infty$. For $n \in \mathbb N$, let $(\text{Hom}(n), \|\cdot\|_F)$ denote the Hilbert space of all homogeneous polynomials of degree $n$ endowed with the Fischer-Fock inner product. Now consider the following diagram:
\[\begin{tikzcd}
\text{Hom}(n) \arrow{r}{C_u} \arrow{d}{I_1} & \text{Hom}(n) \\
(\text{Hom}(n), \|\cdot\|_F) \arrow{r}{C_u}& (\text{Hom}(n), \|\cdot\|_F) \arrow{1-2}{I_2}
\end{tikzcd}
\]
where $I_1$ and $I_2$ are respective identity operators. Clearly, the diagram commutes, and hence, it yields
\beq\label{eq-2}
\|C_u|_{\text{Hom}(n)}\| = \|I_2 C_u|_{(\text{Hom}(n), \|\cdot\|_F)} I_1\| \leqslant \|I_2\| \|C_u|_{(\text{Hom}(n), \|\cdot\|_F)}\| \|I_1\|.
\eeq
Note that $\|C_u|_{(\text{Hom}(n), \|\cdot\|_F)}\| = 1$. Further, 
\beqn
\|I_1 z^\alpha\|_F &=& \|z^\alpha\|_F = \frac{\|z^\alpha\|_{\mathcal H^2(\beta)}}{\|z^\alpha\|_{\mathcal H^2(\beta)}}\|z^\alpha\|_F \leqslant \left(\max_{|\alpha| = n} \frac{\|z^\alpha\|_F}{\|z^\alpha\|_{\mathcal H^2(\beta)}}\right) \|z^\alpha\|_{\mathcal H^2(\beta)}\\ 
&=& \left(\max_{|\alpha| = n} \frac{\sqrt{\alpha!}}{\|z^\alpha\|_{\mathcal H^2(\beta)}}\right) \|z^\alpha\|_{\mathcal H^2(\beta)}.
\eeqn
Similarly, we get
\beqn
\|I_2 z^\alpha\|_{\mathcal H^2(\beta)} \leqslant \left(\max_{|\alpha| = n} \frac{\|z^\alpha\|_{\mathcal H^2(\beta)}}{\sqrt{\alpha!}}\right) \|z^\alpha\|_F.
\eeqn
Putting these values in \eqref{eq-2}, we get
\beqn
\|C_u|_{\text{Hom}(n)}\| \leqslant \max_{|\alpha|=n} \frac{\|z^\alpha\|_{\mathcal H^2(\beta)}}{\sqrt{\alpha!}} \max_{|\delta|=n} \frac{\sqrt{\delta!}}{\|z^\delta\|_{\mathcal H^2(\beta)}}. 
\eeqn
Consequently, we obtain that
\beqn
\|C_u\| = \sup_{n \in \mathbb N} \|C_u|_{\text{Hom}(n)}\| \leqslant \sup_{n \in \mathbb N}\left\{\max_{|\alpha|=n} \frac{\|z^\alpha\|}{\sqrt{\alpha!}} \max_{|\alpha|=n} \frac{\sqrt{\alpha!}}{\|z^\alpha\|}\right\}
\eeqn
which is finite by $(iii)$. Since the quantity on extreme right in the above inequality is independent of $u$, we conclude that the family $\{C_u : u \in \mathcal U(d)\}$ of composition operators on $\mathcal H^2(\beta)$ is uniformly bounded, which proves $(iv)$.
\\

$(iv) \implies (i)$: Follows from Proposition \ref{prop-2.13}. This completes the proof.
\end{proof}

\section{Weakly $\mathcal U(d)$-homogeneous spherically balanced multishifts}

In this section, we refine the results of the preceding section for the class of spherically balanced multishifts. To explain the main result of this section, we recall the definition of balanced multishift from \cite{C-K}. 

Let $\mathscr M_z$ be bounded on $\mathcal H^2(\beta)$. Then following \cite[Definition 1.8]{C-K}, we say that $\mathcal H^2(\beta)$ is {\it spherically balanced} if 
$\beta = \{\beta_\alpha : \alpha \in \mathbb N^d\}$ satisfies
\beqn
\sum_{k=1}^d \frac{\beta^2_{\alpha+\varepsilon_i+\varepsilon_k}}{\beta^2_{\alpha+\varepsilon_i}} = \sum_{k=1}^d \frac{\beta^2_{\alpha+\varepsilon_j+\varepsilon_k}}{\beta^2_{\alpha+\varepsilon_j}}\ \text{ for all }\ \alpha \in \mathbb N^d,\ \text{ and }\ i,j = 1, \ldots, d, 
\eeqn
where $\varepsilon_j$ denotes the $d$-tuple with $1$ at the $j$-th place and zeros elsewhere. The operator tuple $\mathscr M_z$ on the spherically balanced $\mathcal H^2(\beta)$ is referred to as the {\it spherically balanced multishift}. It follows from \cite[Theorem 1.10]{C-K} that every spherically balanced $\mathcal H^2(\beta)$ admits a slice representation $[\mu,\ \mathcal H^2(\gamma)]$, with $\mu$ being a finite positive Reinhardt measure supported on $\partial\mathbb B_d$ and $\mathcal H^2(\gamma)$ being a Hilbert space of formal power series in one variable, in the sense that
\beqn
\|f\|^2_{\mathcal H^2(\beta)} = \int_{\partial\mathbb B_d} \|f_z\|^2_{\mathcal H^2(\gamma)} d\mu(z), \quad f \in \mathcal H^2(\beta),
\eeqn
where $f_z \in \mathcal H^2(\gamma)$ is referred to as the {\it slice of $f$} at $z \in \mathbb C^d$, and is given by
\beqn
f_z(t) = f(tz), \quad t \in \mathbb C. 
\eeqn
In particular, for every $\alpha \in \mathbb N^d$, we have
\beqn
\beta_\alpha = \gamma_{|\alpha|} \|z^\alpha\|_{L^2(\partial\mathbb B_d,\, \mu)}, 
\eeqn
where $\gamma_n = \|t^n\|_{\mathcal H^2(\gamma)}$, $n \in \mathbb N$.

We proceed with the following elementary lemma which provides a sufficient condition for $C_u$ to be bounded on the spherically balanced $\mathcal H^2(\beta)$.

\begin{lemma}\label{lem-Cu-RD}
Let $\mathcal H^2(\beta)$ be spherically balanced with the slice representation $[\mu,\ \mathcal H^2(\gamma)]$, where $\mu$ is a Reinhardt measure supported on $\partial\mathbb B_d$ and $\mathcal H^2(\gamma)$ is a Hilbert space of formal power series in one variable. Let $u \in \mathcal U(d)$. If the Radon-Nikodym derivative $\frac{d\mu\circ u^{-1}}{d\mu}$ is essentially bounded, then $C_u$ is bounded on $\mathcal H^2(\beta)$. 
\end{lemma}

\begin{proof}
Suppose that the Radon-Nikodym derivative $\frac{d\mu\circ u^{-1}}{d\mu}$ is essentially bounded. Let $f \in \mathcal H^2(\beta)$. Then 
\beqn
\|f\|^2_{\mathcal H^2(\beta)} = \int_{\partial\mathbb B_d} \|f_z\|^2_{\mathcal H^2(\gamma)} d\mu(z).
\eeqn
Fix $u \in \mathcal U(d)$. Note that $(f\circ u)_z(t) = (f\circ u)(tz) = f(tu\cdot z) = f_{u \cdot z}(t)$, $z \in \mathbb C^d,\ t \in \mathbb C$. Thus, we have
\beqn
\|f\circ u\|^2_{\mathcal H^2(\beta)} &=& \int_{\partial\mathbb B_d} \|(f\circ u)_z\|^2_{\mathcal H^2(\gamma)} d\mu(z) = \int_{\partial\mathbb B_d} \|f_{u\cdot z}\|^2_{\mathcal H^2(\gamma)} d\mu(z)\\
&=& \int_{\partial\mathbb B_d} \|f_z\|^2_{\mathcal H^2(\gamma)} d\mu(u^{-1}\cdot z) = \int_{\partial\mathbb B_d} \|f_z\|^2_{\mathcal H^2(\gamma)} \frac{d\mu\circ u^{-1}}{d\mu}(z) d\mu(z)\\
&\leqslant& \Big\| \frac{d\mu\circ u^{-1}}{d\mu}\Big\|_{L^\infty(\mu, \partial\mathbb B_d)} \|f\|^2_{\mathcal H^2(\beta)}.
\eeqn
Thus, $C_u$ is bounded on $\mathcal H^2(\beta)$.
\end{proof}

Let $\mu$ be a finite positive Reinhardt measure supported on $\partial\mathbb B_d$ and $\mathcal H^2(\mu)$ be the closure of polynomials in $L^2(\partial\mathbb B_d, \mu)$. Then the monomials are orthogonal in $\mathcal H^2(\mu)$. In other words, $\mathcal H^2(\mu)$ is a special case of $\mathcal H^2(\beta)$. The following result is an intermediate step in the proof of the main result which also refines Theorem \ref{Cu-bdd}. In particular, it shows that $\mathscr M_z$ on $\mathcal H^2(\mu)$ is weakly $\mathcal U(d)$-homogeneous if and only if it is similar to the Szeg\"o shift.

\begin{theorem}\label{wk-hom-H2mu}
Let $\mu$ be a finite positive Reinhardt measure supported on $\partial\mathbb B_d$ and $\mathcal H^2(\mu)$ be the closure of polynomials in $L^2(\partial\mathbb B_d, \mu)$. Suppose that $\mathscr M_z$ is bounded on $\mathcal H^2(\mu)$. Then the following statements are equivalent:
\begin{itemize}
\item[$(i)$] $\mathscr M_z$ is weakly $\mathcal U(d)$-homogeneous.
\item[$(ii)$] There exist positive constants $k$ and $K$ such that for all $u \in \mathcal U(d)$, 
\[k \leqslant \frac{d\mu\circ u^{-1}}{d\mu} (z) \leqslant K \ \text{ for $\mu$-almost every }\ z \in \partial\mathbb B_d.\]
\item[$(iii)$] $\mathscr M_z$ is similar to the Szeg\"o shift.
\end{itemize}
\end{theorem}

\begin{proof}
$(i) \implies (ii)$: Suppose that $\mathscr M_z$ is weakly $\mathcal U(d)$-homogeneous. Then it follows from Proposition \ref{prop-2.13} that $C_u$ is bounded on $\mathcal H^2(\mu)$ for all $u \in \mathcal U(d)$. Consequently, $C_u$ is invertible on $\mathcal H^2(\mu)$ for all $u \in \mathcal U(d)$, and it follows from \eqref{Cu-intertwine} that 
\[C_u \mathscr M_{z_j} = (u\cdot \mathscr M_z)_j C_u\ \text{ for all } j = 1, \ldots, d\ \text{ and }\ u \in \mathcal U(d).\]
Consider the map $T_{u^{-1}} : \mathcal H^2(\mu) \to \mathcal H^2(\mu\circ u^{-1})$ given by
\[T_{u^{-1}} f(z) = f(u^{-1} \cdot z), \quad f \in \mathcal H^2(\mu).\]
Then $T_{u^{-1}}$ is a unitary. Indeed, for $f \in \mathcal H^2(\mu)$, we have
\beqn
\|T_{u^{-1}} f\|^2_{\mathcal H^2(\mu\circ u^{-1})} &=& \int_{\partial\mathbb B_d} |f(u^{-1} \cdot z)|^2 d\mu\circ u^{-1}(z) = \int_{\partial\mathbb B_d} |f(w)|^2 d\mu(w)\\
&=& \|f\|^2_{\mathcal H^2(\mu)}.
\eeqn 
Further, it follows from the proof of Lemma \ref{lem-Ud} that
\[T_{u^{-1}} (u \cdot \mathscr M_z)_j = \mathscr M_{z_j} T_{u^{-1}} \ \text{ for all } j=1, \ldots, d.\]
Thus, $u \cdot \mathscr M_z$ on $\mathcal H^2(\mu)$ is unitarily equivalent to $\mathscr M_z$ on $\mathcal H^2(\mu\circ u^{-1})$. Now consider the map $X_u := T_{u^{-1}} C_u : \mathcal H^2(\mu) \to \mathcal H^2(\mu\circ u^{-1})$. Then $X_u$ is invertible and for all $j=1, \ldots, d$, we have
\beqn
X_u \mathscr M_{z_j} = T_{u^{-1}} C_u \mathscr M_{z_j} = T_{u^{-1}} (u\cdot \mathscr M_z)_j C_u = \mathscr M_{z_j} T_{u^{-1}} C_u = \mathscr M_{z_j} X_u.
\eeqn 
Thus, $\mathscr M_z$ on $\mathcal H^2(\mu)$ is similar to $\mathscr M_z$ on $\mathcal H^2(\mu \circ u^{-1})$. It follows from \cite[Theorem 1]{A} that $\mu$ and $\mu \circ u^{-1}$ are mutually absolutely continuous and (taking the cyclic vectors $f=g=1$ in \cite[Theorem 1]{A}) 
\beq\label{mut-abs}
\frac{1}{\|X_u^{-1}\|} \int_{\partial\mathbb B_d} |p(z)|^2 d\mu(z) \leqslant \int_{\partial\mathbb B_d} |p(z)|^2 d\mu\circ u^{-1}(z) \leqslant \|X_u\| \int_{\partial\mathbb B_d} |p(z)|^2 d\mu(z)
\eeq 
for all polynomials $p$. Considering the right inequality in \eqref{mut-abs}, we get
\[\int_{\partial\mathbb B_d} \Big(\|X_u\| - \frac{d\mu\circ u^{-1}}{d\mu}(z)\Big)|p(z)|^2 d\mu(z) \geqslant 0\]
for all polynomials $p$. Appealing to \cite[Theorem 0]{A} and using the fact that characteristic (or indicator) functions can be approximated by positive continuous functions, we obtain that
\[\int_{\Delta} \Big(\|X_u\| - \frac{d\mu\circ u^{-1}}{d\mu}(z)\Big) d\mu(z) \geqslant 0 \] 
for every measurable subset $\Delta \subset \partial\mathbb B_d.$ Consequently, we  get that 
\beq\label{mut-abs-1}
\frac{d\mu\circ u^{-1}}{d\mu}(z) \leqslant \|X_u\| \ \text{ for $\mu$-almost every }\ z \in \partial\mathbb B_d.
\eeq 
Similar arguments while considering the left inequality in \eqref{mut-abs} yields
\beq\label{mut-abs-2}
\frac{1}{\|X_u^{-1}\|} \leqslant \frac{d\mu\circ u^{-1}}{d\mu}(z) \ \text{ for $\mu$-almost every }\ z \in \partial\mathbb B_d.
\eeq 
By \eqref{mut-abs-1} and \eqref{mut-abs-2}, we obtain that
\beq\label{mut-abs-3}
\frac{1}{\|X_u^{-1}\|} \leqslant \frac{d\mu\circ u^{-1}}{d\mu}(z) \leqslant \|X_u\| \ \text{ for $\mu$-almost every }\ z \in \partial\mathbb B_d.
\eeq
Since $T_{u^{-1}}$ is a unitary, we have $\|X_u\| = \|C_u\|$ and $\|X_u^{-1}\| = \|C_{u^*}\|$ for all $u \in \mathcal U(d)$. Substituting these values in \eqref{mut-abs-3}, we get
\beqn
\inf_{u \in \mathcal U(d)}\frac{1}{\|C_{u^*}\|} \leqslant \frac{d\mu\circ u^{-1}}{d\mu}(z) \leqslant \sup_{u \in \mathcal U(d)}\|C_u\| \ \text{ for $\mu$-almost every }\ z \in \partial\mathbb B_d.
\eeqn
Since $\mathscr M_z$ is weakly $\mathcal U(d)$-homogeneous, it follows from Theorem \ref{Cu-bdd} that the family $\{C_u : u \in \mathcal U(d)\}$ of composition operators on $\mathcal H^2(\mu)$ is uniformly bounded. Let $K := \sup_{u \in \mathcal U(d)}\|C_u\|$. Also, note that 
\[\inf_{u \in \mathcal U(d)}\frac{1}{\|C_{u^*}\|} = \frac{1}{\sup_{u \in \mathcal U(d)}\|C_{u^*}\|} = \frac{1}{\sup_{u \in \mathcal U(d)}\|C_u\|} > 0.\]
Let $k := \inf_{u \in \mathcal U(d)}\frac{1}{\|C_{u^*}\|}$. Putting these values in above inequality, we obtain that for all $u \in \mathcal U(d)$, 
\[k \leqslant \frac{d\mu\circ u^{-1}}{d\mu}(z) \leqslant K \ \text{ for $\mu$-almost every }\ z \in \partial\mathbb B_d.\]
This completes the verification of $(ii)$.\\

$(ii) \implies (iii)$: Assume that $(ii)$ holds. For each Borel measurable subset $\Delta$ of $\partial\mathbb B_d$, define
\beqn
\nu(\Delta) := \int_{\mathcal U(d)} \mu\circ u^{-1}(\Delta) du,
\eeqn
where $du$ is the Haar measure on the compact group $\mathcal U(d)$. Since $du$ is left-invariant measure, it follows that $\nu$ is $\mathcal U(d)$-invariant. In fact, for any Borel measurable subset $\Delta$ of $\partial\mathbb B_d$ and $v \in \mathcal U(d)$, we have
\beqn
\nu(v \Delta) = \int_{\mathcal U(d)} \mu \circ u^{-1} (v \Delta) du.
\eeqn   
Let $u^{-1} v = w^{-1}$. Then $u = vw$ and hence, $du = d(vw)$. This gives that
\beqn
\nu(v \Delta) = \int_{\mathcal U(d)} \mu \circ w^{-1}(\Delta) d(vw) = \int_{\mathcal U(d)} \mu \circ w^{-1} (\Delta) dw = \nu(\Delta).
\eeqn   
Now we show that $\mu$ and $\nu$ are mutually absolutely continuous. To see this, observe that if $\mu(\Delta) = 0$ for some Borel measurable subset $\Delta$ of $\partial\mathbb B_d$, then $(ii)$ implies that $\mu \circ u^{-1}(\Delta) = 0$, and hence, $\nu(\Delta) = 0$. Conversely, assume that $\nu(\Delta) = 0$ for some Borel measurable subset $\Delta$ of $\partial\mathbb B_d$. Then we have
\beqn
0 &=& \nu(\Delta) = \int_{\mathcal U(d)} \mu\circ u^{-1}(\Delta) du = \int_{\mathcal U(d)} \int_\Delta d\mu\circ u^{-1}(z) du \\
&=& \int_{\mathcal U(d)} \int_\Delta \frac{d\mu\circ u^{-1}}{d\mu}(z) d\mu(z) du \geqslant k \mu(\Delta) \int_{\mathcal U(d)} du,
\eeqn
where the last inequality follows from $(ii)$. This gives that $\mu(\Delta) = 0$. Thus, $\mu$ is mutually absolutely continuous with respect to a $\mathcal U(d)$-invariant measure $\nu$. By uniqueness of $\mathcal U(d)$-invariant measure supported on $\partial\mathbb B_d$, we conclude that $\nu = c \sigma$ for some $c > 0$ ($c = 1$ if $\nu$ is normalized). Thus, $\mu$ is mutually absolutely continuous with respect to $\sigma$.

Further, by Fubini–Tonelli theorem, we obtain that 
\beqn
\nu(\Delta) &=& \int_{\mathcal U(d)} \mu\circ u^{-1}(\Delta) du = \int_{\mathcal U(d)} \int_\Delta d\mu\circ u^{-1}(z) du \\
&=& \int_{\mathcal U(d)} \int_\Delta \frac{d\mu\circ u^{-1}}{d\mu}(z) d\mu(z) du  = \int_\Delta \int_{\mathcal U(d)} \frac{d\mu\circ u^{-1}}{d\mu}(z) du\, d\mu(z) 
\eeqn
for every Borel measurable subset $\Delta$ of $\partial\mathbb B_d$. This gives that
\beqn\label{Fubini-nu-mu}
\frac{d\nu}{d\mu}(z) = \int_{\mathcal U(d)} \frac{d\mu\circ u^{-1}}{d\mu}(z) du.
\eeqn
Using $(ii)$ together with the fact that $\nu = c \sigma$, we obtain that
\beqn\label{nu-sigma-eq}
k_1 \leqslant \frac{d\sigma}{d\mu}(z) \leqslant K_1\ \text{ for $\mu$-almost every }\ z \in \partial\mathbb B_d,
\eeqn
for some positive constants $k_1$ and $K_1$. Consequently, we get
\beqn
k_1 \leqslant \frac{\|z^\alpha\|_{\mathcal H^2(\sigma)}}{\|z^\alpha\|_{\mathcal H^2(\mu)}} \leqslant K_1\ \text{ for all }\ \alpha \in \mathbb N^d.
\eeqn
Since $\mathscr M_z$ on $\mathcal H^2(\sigma)$ is the Szeg\"o shift, it follows from \cite[Lemma 2.2]{K} that $\mathscr M_z$ on $\mathcal H^2(\mu)$ is similar to the Szeg\"o shift, which verifies $(iii)$.

That $(iii) \implies (i)$ is obvious as the Szeg\"o shift is $\mathcal U(d)$-homogeneous. This completes the proof.
\end{proof}

Now we present the main result of this section.

\begin{theorem}
Let $\mathcal H^2(\beta)$ be spherically balanced with the slice representation $[\mu,\ \mathcal H^2(\gamma)]$, where $\mu$ is a Reinhardt measure supported on $\partial\mathbb B_d$ and $\mathcal H^2(\gamma)$ is a Hilbert space of formal power series in one variable. Suppose that $\mathscr M_z$ on $\mathcal H^2(\beta)$ is bounded. Then $\mathscr M_z$ is weakly $\mathcal U(d)$-homogeneous if and only if $\mathscr M_z$ on $\mathcal H^2(\beta)$ is similar to $\mathscr M_z$ on $\mathcal H^2(\tilde\beta)$ with slice representation $[\sigma,\ \mathcal H^2(\gamma)]$.
\end{theorem}

\begin{proof}
Suppose that $\mathscr M_z$ on $\mathcal H^2(\beta)$ is weakly $\mathcal U(d)$-homogeneous. Then by Proposition \ref{prop-2.13}, $C_u$ is bounded and invertible on $\mathcal H^2(\beta)$ for all $u \in \mathcal U(d)$. Now fix $u \in \mathcal U(d)$, and consider $\mathcal H_{u^{-1}}^2(\beta)$. Then for any polynomial $p \in \text{Hom}(n)$, $n \in \mathbb N$, we observe that
\beqn
\|p\circ u^{-1}\|^2_{\mathcal H_{u^{-1}}^2(\beta)} = \|p\|^2_{\mathcal H^2(\beta)} = \gamma_n^2 \|p\|^2_{L^2(\partial\mathbb B_d,\, \mu)} = \gamma_n^2 \int_{\partial\mathbb B_d} |p(z)|^2 d\mu(z).
\eeqn
Since $p \in \text{Hom}(n) \Leftrightarrow p\circ u \in \text{Hom}(n)$, the above equality yields that
\beqn
\|p\|^2_{\mathcal H_{u^{-1}}^2(\beta)} &=& \gamma_n^2 \int_{\partial\mathbb B_d} |p(u \cdot z)|^2 d\mu(z) = \gamma_n^2 \int_{\partial\mathbb B_d} |p(w)|^2 d\mu\circ u^{-1}(w)\\
&=& \gamma_n^2 \|p\|^2_{L^2(\partial\mathbb B_d,\, \mu\circ u^{-1})}.
\eeqn
This gives us a slice representation of $\mathcal H_{u^{-1}}^2(\beta)$. Note that the map $T_{u^{-1}} : \mathcal H^2(\beta) \rar \mathcal H_{u^{-1}}^2(\beta)$, defined as $T_{u^{-1}} f(z) = f(u^{-1} \cdot z)$ for all $f \in \mathcal H^2(\beta)$, is a unitary. Thus, the map $X_u := T_{u^{-1}} C_u : \mathcal H^2(\beta) \rar \mathcal H_{u^{-1}}^2(\beta)$ is a bounded invertible operator. Since 
\beqn 
\mathcal H^2(\beta) = \bigoplus_{n\in \mathbb N} \text{Hom}(n) \ \text{ and }\ \mathcal H_{u^{-1}}^2(\beta) = \bigoplus_{n\in \mathbb N} \text{Hom}(n),
\eeqn
it can be easily seen that $X_u$ is a block diagonal operator with respect to above decompositions. Hence, we obtain that 
\beqn
\|X_u p\|^2_{\mathcal H_{u^{-1}}^2(\beta)} \leqslant \|X_u\| \|p\|^2_{\mathcal H^2(\beta)} \ \text{ for all }\ p \in \text{Hom}(n),\ n \in \mathbb N. 
\eeqn
Consequently, we get
\beqn
\gamma_n^2 \|X_u p\|^2_{L^2(\partial\mathbb B_d,\, \mu\circ u^{-1})} \leqslant \gamma_n^2 \|X_u\| \|p\|^2_{L^2(\partial\mathbb B_d,\, \mu)},
\eeqn
which, in turn, implies that
\beq\label{Xu-H2mu}
\|X_u p\|^2_{L^2(\partial\mathbb B_d,\, \mu\circ u^{-1})} \leqslant \|X_u\| \|p\|^2_{L^2(\partial\mathbb B_d,\, \mu)} \ \text{ for all }\ p \in \text{Hom}(n),\ n \in \mathbb N. 
\eeq
Let $\mathcal H^2(\mu)$ be the closure of polynomials in $L^2(\partial\mathbb B_d,\, \mu)$. Since $\mu$ is Reinhardt, the monomials are orthogonal in $\mathcal H^2(\mu)$. Therefore,  
\beqn 
\mathcal H^2(\mu) = \bigoplus_{n\in \mathbb N} \text{Hom}(n) \ \text{ and }\ \mathcal H^2(\mu \circ u^{-1}) = \bigoplus_{n\in \mathbb N} \text{Hom}(n).
\eeqn
Note that \eqref{Xu-H2mu} gives that $X_u : \mathcal H^2(\mu) \rar \mathcal H^2(\mu \circ u^{-1})$ is bounded, and similar arguments for $X_u^{-1} : \mathcal H_{u^{-1}}^2(\beta) \rar \mathcal H^2(\beta)$ yield that $X_u : \mathcal H^2(\mu) \rar \mathcal H^2(\mu \circ u^{-1})$ is invertible. Now proceeding as in the proof of Theorem \ref{wk-hom-H2mu}, we obtain that $\mathscr M_z$ on $\mathcal H^2(\mu)$ is similar to the Szeg\"o shift, that is $\mathscr M_z$ on $\mathcal H^2(\sigma)$.  
Now consider 
\beqn
\tilde\beta_\alpha := \gamma_{|\alpha|} \|z^\alpha\|_{L^2(\partial\mathbb B_d,\, \sigma)}\ \text{ for all }\ \alpha \in \mathbb N^d.
\eeqn 
Then it is easy to see that $\mathscr M_z$ on $\mathcal H^2(\tilde\beta)$ is $\mathcal U(d)$-homogeneous. Further, it follows from \cite[Theorem 2.1]{K} that $\mathscr M_z$ on $\mathcal H^2(\beta)$ is similar to $\mathscr M_z$ on $\mathcal H^2(\tilde\beta)$. Clearly, $\mathcal H^2(\tilde\beta)$ has the slice representation $[\sigma,\ \mathcal H^2(\gamma)]$. The converse is obvious. This completes the proof. 
\end{proof}

\end{document}